\documentclass[final]{siamltex}

\usepackage{subfigure}
\usepackage{graphicx}
\usepackage{amsmath}
\usepackage{amssymb}
\usepackage{algorithm}
\usepackage{algorithmic}

\newcommand{\myup}{\mathrm{up}}
\newcommand{\norm}[1]{\left \lVert #1 \right \rVert}
\newcommand{\delA}{\Delta A}
\newcommand{\ellDelta}{\tilde{\ell}}
\newcommand{\delQ}{\tilde{Q}}

\newcommand{\delX}{\tilde{X}}

\newcommand{\range}[1]{\mathrm{range}\left(#1\right)}
\newcommand{\sr}[1]{\mathrm{sr}\left(#1\right)}

\renewcommand{\tilde}[1]{\widetilde{#1}}

\title{Conditioning of Leverage Scores and Computation by QR Decomposition\thanks{The first author was supported in part by 
Department of Education Grant P200A090081. The second and third authors were supported 
in part by NSF grant CCF-1145383. The second author also
acknowledges the support from the XDATA Program of the Defense Advanced
Research Projects Agency (DARPA), administered through Air Force
Research Laboratory contract FA8750-12-C-0323 FA8750-12-C-0323.
All three authors performed part of the work with support from the National
Science Foundation under Grant DMS-1127914 to the Statistical and
Applied Mathematical Sciences Institute.}
}
\author{
John T. Holodnak\thanks{Department of Mathematics, North Carolina State University, P.O. Box 8205, Raleigh, NC 27695-8205, USA (\texttt{jtholodn@ncsu.edu}, \texttt{http://www4.ncsu.edu/{\char'176}jtholodn/})}
\and
Ilse C. F. Ipsen\thanks{%
Department of Mathematics, North Carolina State University, P.O. Box 8205,
Raleigh, NC 27695-8205, USA (\texttt{ipsen@ncsu.edu}, 
\texttt{http://www4.ncsu.edu/{\char'176}ipsen/})}
\and
Thomas Wentworth\thanks{%
Koch Institute, Massachusetts Institute of Technology, 77 Massachusetts Ave, Cambridge, 
MA 02139, USA (\texttt{tawentwo@mit.edu})}
}
\begin{document}
\maketitle

\begin{abstract} 
The leverage scores of a full-column rank matrix $A$ are the squared row norms of any orthonormal basis for $\range{A}$. We show that corresponding leverage scores of two matrices 
$A$ and $A+\delA$ are close in the relative sense, if they have large magnitude and if
all principal angles between the column spaces of $A$ and $A+\delA$ are small.

We also show three classes of bounds that are based on perturbation results of QR 
decompositions.
They demonstrate that relative differences between individual leverage scores strongly 
depend on the particular
type of perturbation $\delA$. The bounds imply that the relative accuracy of an 
individual leverage score depends on: its magnitude and the two-norm condition of $A$, 
if $\delA$ is a general perturbation; the two-norm condition number of $A$,
if $\delA$  is a perturbation
with the same norm-wise row-scaling as $A$;  (to first order) neither condition number 
nor leverage score magnitude, if $\delA$ is a component-wise row-scaled perturbation.
Numerical experiments confirm the qualitative and quantitative accuracy of our bounds.
\end{abstract}


\begin{AM} 
65F25, 65F35, 62J20, 68W20, 15A12, 15A23
\end{AM}

\section{Introduction} 
Leverage scores are scalar quantities associated with the column space of a matrix, and 
can be computed from the rows of \textit{any} orthonormal basis for this space. 

\subsubsection*{Leverage scores}
We restrict  our discussion here to leverage scores of full column rank matrices.

\begin{definition}\label{d_lev}
Let $A$ be a real $m \times n$ matrix with $rank(A)=n$. If $Q$ is any $m\times n$ matrix 
whose columns form an orthonormal basis for $\range{A}$,  then the leverage scores of $A$ are
 $$\ell_j \equiv \norm{e_j^T Q}^2_2, \qquad 1\leq j \leq m.$$
 Here $e_j$ denotes the $j$th column of the $m\times m$ identity matrix,
 and $e_j^TQ$ denotes the $j$th row of $Q$.
\end{definition}

Note that leverage scores are independent of the orthonormal basis, since
$$
\norm{e_j^TQ}_2^2 = e_j^TQQ^Te_j = (QQ^T)_{jj}, \quad 1 \leq j \leq m
$$
and $QQ^T$ is the unique orthogonal projector onto $\range{A}$.

The basic  properties of leverage scores are 
$$0\leq \ell_j\leq 1,  \qquad 1\leq j\leq m, \qquad \mathrm{and} \qquad 
\sum_{j=1}^m{\ell_j}=n.$$
Hoaglin and Welsch  introduced statistical leverage scores in 1978 to detect outliers in regression problems \cite[Section 2]{HoagW78}, \cite[Section 5.1]{IpW13}, 
\cite[Section 2.2]{VelleW81}. About thirty years later, Mahoney, Drineas and their coauthors
started to advocate the use of leverage scores in randomized matrix algorithms \cite{DMM06a, DrineasMM06, Mah11}.
More specifically, leverage scores are the basis for importance sampling strategies, 
in the context of
low rank approximations \cite{DrineasMM06}, CUR decompositions
\cite{DrineasMM08}, subset selection \cite{BMD10}, Nystr\"{o}m approximations \cite{TalR10}, 
least squares problems \cite{DMM06a}, and matrix completion \cite{CanR09}, to name just a few.
Leverage scores also play a crucial role in the analysis of randomized algorithms \cite{IpW13}, and 
fast algorithms have been developed for their approximation \cite{DMMW2012,LMP2013,Magdon2010}.

\subsubsection*{Motivation}
Since leverage scores depend only on the column space, and are not tied to any particular 
orthonormal basis, the question is how to compute them. Many existing papers, among them
the survey monograph \cite[Definition 1]{Mah11},
define leverage scores as row norms of a thin left singular vector matrix.
However, the sensitivity of singular vectors is determined by the corresponding singular value gaps.

This, and the fact that QR decompositions, when implemented via Householder transformations
or Givens rotations, are numerically stable \cite[Sections 19.1--19.7]{Higham2002},
motivated us to investigate QR decompositions for the computation of leverage scores.  In this paper, we derive bounds on the difference between the leverage scores of a matrix $A$ and a perturbation $A+\delA$, when the leverage scores are computed from a QR decomposition.  Note that we do not assume a particular implementation of the QR decomposition and assume that quantities are computed in exact arithmetic.  We consider our results to be a first step towards determining whether computing leverage scores with a QR decomposition is numerically stable.  Since most of our bounds do not exploit the zero structure of the upper triangular factor, they can be readily extended to polar decompositions.

\subsection{Overview}
We present a short overview of the contents of the paper and the main results. 
For brevity, we display only the first order terms in the bounds, and omit the technical
assumptions.

\subsubsection*{Notation}
The $m\times m$ identity matrix is $I_m$, with columns $e_j$ and rows $e_j^T$,
$1\leq j\leq m$.

For a real $m\times n$ matrix $A$ with $\rank(A)=n$, the two-norm condition 
number with respect to 
left inversion is $\kappa_2(A)\equiv \|A\|_2\|A^{\dagger}\|_2$, where $A^{\dagger}$ is the
Moore-Penrose inverse. The stable rank is $\sr{A}\equiv \|A\|_F^2/\|A\|_2^2$, where
$\sr{A}\leq \rank(A)$.

We denote the leverage scores of a perturbed matrix $A+\delA$ by $\ellDelta_j$ and refer to the quantities $|\ellDelta_j-\ell_j|$ and $|\ellDelta_j-\ell_j|/\ell_j$ as the the absolute leverage score difference and relative leverage score difference, respectively.  We assume, tacitly, that relative leverage score difference bounds $|\ellDelta_j-\ell_j|/\ell_j$ apply only for $\ell_j>0$.

\subsubsection*{Conditioning of leverage scores (Section~\ref{s_angles})}
Before even thinking about computation, we need to determine the 
conditioning of individual leverage scores. To this end, let 
 $A$ and $A+\delA$ be real $m\times n$ matrices with $\rank(A)=\rank(A+\delA)=n$,
and leverage scores $\ell_j$ and $\ellDelta_j$, $1\leq j\leq m$, respectively.
We show (Corollary~\ref{c_1}) that the relative sensitivity of individual leverage scores 
to subspace rotations is determined by their magnitude.
That is,  if $\theta_n$ is the largest principal angle between $\range{A}$ and $\range{A+\delA}$,
then
$$\frac{|\ellDelta_j - \ell_j|}{\ell_j}  \leq 
2\>\sqrt{\frac{1-\ell_j}{\ell_j}}\>\sin{\theta_n} + \mathcal{O}\left((\sin{\theta_n})^2\right),
\qquad 1\leq j\leq m.$$
Thus, large leverage scores tend to be better conditioned, in the relative sense, to subspace
rotations than small leverage scores.

The same holds for general perturbations in the two-norm (Theorem~\ref{t_2}). 
If $\epsilon = \|\delA\|_2/\|A\|_2$ is the two-norm of the perturbation, then
$$\frac{|\ellDelta_j - \ell_j|}{\ell_j}  \leq 
2\>\sqrt{\frac{1-\ell_j}{\ell_j}}\>\kappa_2(A)\>\epsilon + \mathcal{O}(\epsilon^2), \qquad
1\leq j\leq m.$$
Therefore, all leverage scores are ill-conditioned under general norm-wise perturbations,
if $A$ is ill-conditioned with respect to inversion; in addition,
larger leverage scores are better conditioned than smaller ones.

A bound similar to the one above holds also for projected perturbations
$\epsilon^{\perp} =\|\mathcal{P}^{\perp}\,\delA\|_2/\|A\|_2$,
where $\mathcal{P}^{\perp}=I_m-AA^{\dagger}$ is the orthogonal projector onto the
orthogonal complement of $\range{A}$. The projected perturbations
remove the contribution of $\delA$ that lies in $\range{A}$. This is important when $\epsilon$
is large, but $\delA$ has only a small contribution in $\range{A}^{\perp}$.
Note that $\delA$ does not change the leverage scores if $\range{A+\delA}=\range{A}$.

\subsubsection*{Leverage scores computed with a QR decomposition (Section~\ref{s_qrpert})}
With the conditioning of individual leverage scores now established, we present perturbation bounds that represent the first 
step in assessing the numerical stability of the QR decomposition for computing leverage scores.

\paragraph{Section~\ref{s_norm}}
Our first result is a bound derived from existing QR perturbation results
that make no reference to a particular implementation. If
$\epsilon_F=\|\delA\|_F/\|A\|_F$ is the total mass of the perturbation, then
the leverage scores $\ellDelta_j$ computed from a QR decomposition of $A+\delA$ satisfy 
$$\frac{|\ellDelta_j - \ell_j|}{\ell_j}  \leq 12\>\sqrt{\frac{1-\ell_j}{\ell_j}} \>\sr{A}^{1/2}\>\kappa_2(A)\>
\epsilon_F +\mathcal{O}(\epsilon_F^2), \qquad 1\leq j\leq m.$$
Therefore, if $\delA$ is a general matrix perturbation, then leverage scores, computed from a QR
decomposition of $A+\delA$ are well-conditioned in the norm-wise  sense, provided they 
have large magnitude and $A$ is well-conditioned.

\paragraph{Section~\ref{s_row}}
The next bound is derived from scratch and does not rely on existing QR perturbation results.
Again, it makes no assumptions on the matrix perturbation
$\delA$, but is  able to recognize norm-wise row-scaling in $\delA$. If
$\epsilon_j= \|e_j^T\delA\|_2/\|e_j^TA\|_2$,  $1\leq j\leq m$, are norm-wise perturbations of the
rows of~$A$, then the leverage scores $\ellDelta_j$ computed
from a QR decomposition of $A + \delA$ satisfy
$$\frac{\left| \ellDelta_j - \ell_j \right|}{\ell_j} \leq 
2\>\left(\epsilon_j + \sqrt{2}\,\sr{A}^{1/2}\,\epsilon_F\right)\>\kappa_2(A) + 
\mathcal{O}(\epsilon_F^2), \qquad 1\leq j\leq m.$$
The perturbation $\epsilon_j$ represents the \textit{local} effect of $\delA$, because it indicates
how the $j$th relative leverage score difference depends on
the perturbation in row $j$ of~$A$.
In contrast, $\epsilon_F$, containing the total mass of the perturbation,
represents the \textit{global} effect on all leverage scores.

A similar bound holds for projected perturbations 
$\epsilon^{\perp}_F=\|\mathcal{P}^{\perp}\,A\|_F/\|A\|_F$ and
$\epsilon^{\perp}_j=\|e_j^T\mathcal{P}^{\perp}\,\delA\|_2/\|e_j^TA\|_2$,  $1\leq j\leq m$.

\paragraph{Section~\ref{s_comp}}
The natural follow up question is:
What if $\delA$ does indeed represent a row-scaling of $A$? Can we get tighter bounds?
The answer is yes.
If $|e_j^T \delA|\leq \eta_j\> |e_j^TA|$, $1\leq j\leq m$, with 
$\eta =\max_{1\leq j\leq m}{\eta_j}$,
are component-wise row-scaled perturbations, then the leverage scores $\ellDelta_j$ 
computed from a QR decomposition of $A+\delA$ satisfy
$$\frac{\left| \ellDelta_j - \ell_j \right|}{\ell_j} \leq 2\>\left( \eta_j
+ \sqrt{2}\,n\>\eta\right) + \mathcal{O}(\eta^2), \qquad 1\leq j\leq m.$$
Thus, under component-wise row-scaled perturbations,  
leverage scores computed with a QR decomposition have
relative leverage score differences that depend, to first order, neither on the condition number nor on
the magnitudes of the leverage scores. 

\subsubsection*{Numerical experiments (Sections~\ref{s_angles} and \ref{s_qrpert})}
After each of the bounds presented in Sections~\ref{s_angles} and \ref{s_qrpert}, we perform numerical experiments that illustrate that the 
bounds correctly capture the
relative leverage score differences under different types of perturbations.

\subsubsection*{Summary (Section~\ref{s_sum})}
We summarize the results in this paper and describe a few directions for future research.

\subsubsection*{Appendix (Section~\ref{s_proofs})}
We present the proofs for all results in Sections \ref{s_angles} and~\ref{s_qrpert}.

\section{Conditioning of leverage scores}\label{s_angles}
We determine the absolute and relative sensitivity of leverage scores to rotations 
of the column space
(Section~\ref{s_a1}), and to general matrix perturbations in the two-norm (Section~\ref{s_a2}).

\subsection{Principal angles}\label{s_a1}
We show that the leverage scores of two matrices are close in the absolute sense, if the all angles between their column spaces are small (Theorem~\ref{t_1}). Larger leverage scores 
tend to better conditioned in the relative sense (Corollary~\ref{c_1}).

Principal angles between two subspaces quantify the distance between the spaces in ``every dimension.''

\begin{definition}[Section 6.4.3 in \cite{GovL13}]\label{d_angles}
Let $A$ and $\delA$ be real $m \times n$ matrices with $\rank(A)=\rank(A + \delA)=n$,
and let $Q$ and $\delQ$ be orthonormal bases for $\range{A}$ and $\range{A+\delA}$,
respectively.  

Let $Q^T\delQ = U\Sigma V^T$ be a  SVD, where $U$ and $V$ are $n \times n$ 
orthogonal matrices, and 
$\Sigma=\diag\begin{pmatrix}\cos{\theta_1} & \cdots & \cos{\theta_n}\end{pmatrix}$ is a
$n\times n$ diagonal matrix with $1\geq \cos{\theta_1}\geq \cdots \geq \cos{\theta_n}\geq 0$.
Then,
$0 \leq \theta_1 \leq \cdots \leq \theta_n \leq \pi /2$ are the  \textit{principal angles} 
between the column spaces of $\range{A}$ and $\range{A + \delA}$. 
\end{definition}

Below we bound the absolute leverage score difference in terms of the largest and smallest 
principal angles. 
 
\begin{theorem}[Absolute leverage score difference]\label{t_1}
Let $A$ and $\delA$ be real $m \times n$ matrices, with $\rank(A)=\rank(A+\delA)=n$. Then,
$$|\ellDelta_j - \ell_j|  \leq 2\>\sqrt{\ell_j(1-\ell_j)}\>\cos{\theta_1}\sin{\theta_n} + (\sin{\theta_n})^2,
\qquad 1\leq j\leq m.$$
If, in addition $m=2n$, then also
$$1-\left(\sin{\theta_n}\>\sqrt{\ell_j}+\cos{\theta_1}\>\sqrt {1-\ell_j}\right)^2\leq \ellDelta_j
\leq \left(\cos{\theta_1}\>\sqrt{\ell_j}+\sin{\theta_n}\>\sqrt{1-\ell_j}\right)^2.$$
\end{theorem}

\begin{proof}
See Section \ref{s_t1proof}.
\end{proof}

Theorem \ref{t_1} implies
that the leverage scores of $A$ and $A + \delA$ are close in the absolute sense, if the principal
angles between their column spaces are small.  Theorem~\ref{t_1} 
holds with equality if $A$ and $A + \delA$ have 
the same column space, because then the smallest angle $\theta_n$ is zero,  and so is the bound. 

In the special case $m=2n$, better bounds are possible because
$\range{A}$ and its orthogonal complement $\range{A}^{\perp}$
have the same dimension. In addition to implying
$\ellDelta_j=\ell_j$ for $\range{A}=\range{A+\delA}$,
Theorem~\ref{t_1} also implies  $\ellDelta_j=1-\ell_j$ for
$\range{A+\Delta A}=\range{A}^{\perp}$, $1\leq j\leq m$.

Next is a bound for the relative leverage score difference in terms of principal angles.

\begin{corollary}[Relative leverage score difference]\label{c_1}
Under the conditions of Theorem~\ref{t_1}, 
$$\frac{|\ellDelta_j - \ell_j|}{\ell_j}  \leq 
2\>\sqrt{\frac{1-\ell_j}{\ell_j}}\>\cos{\theta_1}\sin{\theta_n} + \frac{(\sin{\theta_n})^2}{\ell_j},
\qquad 1\leq j\leq m.$$
\end{corollary}

Corollary~\ref{c_1} implies that the relative sensitivity of leverage scores to rotations of
$\range{A}$ depends on the magnitude of the leverage scores. In particular, larger
leverage scores tend to be better conditioned. 

\subsubsection*{Numerical experiments: Figure~\ref{f_fig1}}
We illustrate the effect of subspace rotations 
on the relative leverage score differences $|\ellDelta_j-\ell_j|/\ell_j$.  We compute the leverage scores of a matrix $A$ and a perturbation $A+\delA$.
The matrix $A$ has dimension $1000\times 25$, $\kappa_2(A)=1$, and 
leverage scores that increase in four steps,  from $10^{-10}$ to about~$10^{-1}$,
see Figure~\ref{f_fig1}(a). It is generated with the Matlab commands
\begin{eqnarray}\label{e_A1}
 A1 &=& \diag\begin{pmatrix}I_{250} & 10^2\,I_{250} & 10^3\, I_{250} & 10^{4}\, I_{250}
 \end{pmatrix}\>\mathsf{randn(1000,25)} \\[0pt] 
[A,\sim] &=&\mathsf{qr}(A1,0). \nonumber
 \end{eqnarray}
The leverage scores of the perturbed matrix $A+\delA$ are computed with the 
MATLAB QR decomposition  $\mathsf{qr(A+\delA,0)}$.

The perturbations in the following sections are chosen so that they are large enough to dominate the round off errors.

\begin{figure}
\begin{center}
\includegraphics[width=5.5in]{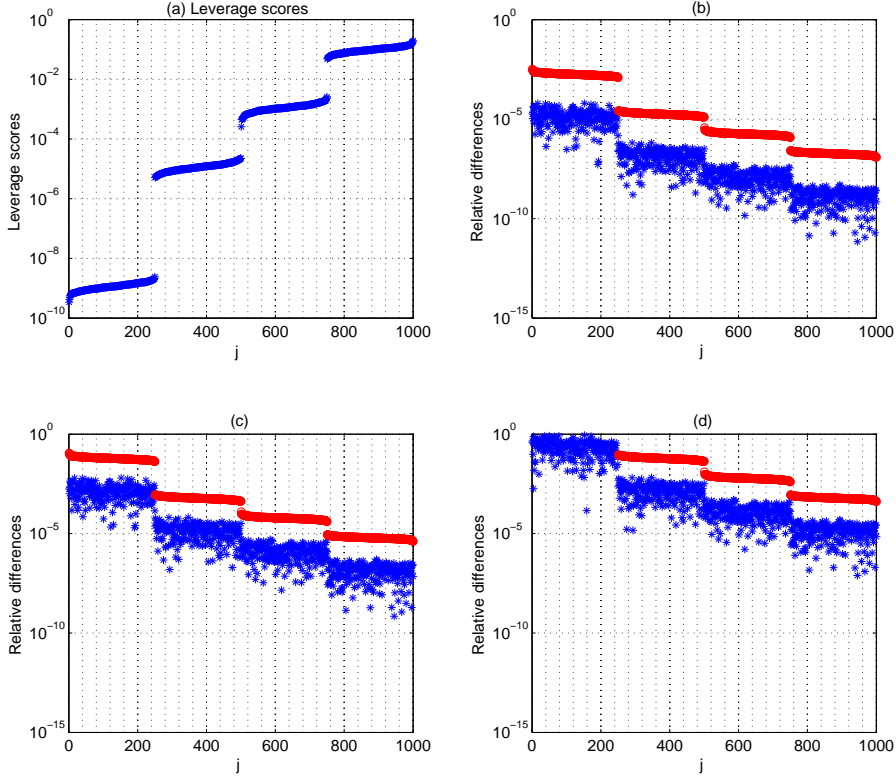}
\end{center}
\caption{(a) Leverage scores $\ell_j$ for matrices in (\ref{e_A1}); (b)--(d) relative leverage score differences 
$|\ellDelta_j-\ell_j|/\ell_j$ (blue stars) and bound for Corollary~\ref{c_1} 
(red line above the stars) vs index~$j$ for $\sin{\theta_n}=10^{-8}$ (b),
 $10^{-6}$ (c), and $10^{-4}$ (d). 
}\label{f_fig1}
\end{figure}

Figure~\ref{f_fig1}(b)--(d) shows the relative sensitivities and the bound from Corollary~\ref{c_1},
for perturbations due to increasing principal angles $\sin{\theta_n}\approx 10^{-8}, 10^{-6}, 10^{-4}$. The relative leverage score differences decrease with the same 
step size with which the leverage score magnitude increases. 

In Figure~\ref{f_fig1}(b), where $\sin{\theta_n}\approx 10^{-8}$, the relative 
leverage score differences decrease from $10^{-5}$ for the smallest leverage scores to about $10^{-9}$ for the largest
leverage scores. The differences are larger by a factor of 100 in Figure~\ref{f_fig1}(c), and again in 
Figure~\ref{f_fig1}(d), where the 250 smallest leverage scores have lost all  accuracy because
they are dominated by the perturbation $\sin{\theta_n}$.
Thus, the relative changes in leverage scores are proportional to $\sin{\theta_n}$.

Corollary~\ref{c_1} shows the same qualitative behavior
as the leverage score differences. The bound decreases with the leverage score magnitude, and
overestimates the worst case differences by a factor of about 100. Thus, 
Corollary~\ref{c_1} represents a realistic estimate for the relative
conditioning of the leverage scores to changes in principal angles.

\subsection{General matrix perturbations in the two-norm}\label{s_a2}
From the bounds for principal angles in Section~\ref{s_a1}, we derive 
bounds for the relative leverage score differences in terms of general perturbations in the two-norm,
$$\epsilon\equiv \frac{\|\delA\|_2}{\|A\|_2}, \qquad
\epsilon^{\perp}\equiv \frac{\|(I_m-AA^{\dagger})\,\delA\|_2}{\|A\|_2}.$$
The second perturbation removes the contribution of $\delA$ that lies in $\range{A}$. 
This is important when $\epsilon$ is large, but $\delA$ has only a small contribution in $\range{A}$.
Note that $\delA$ does not change the leverage scores if $\range{A+\delA}=\range{A}$.
  
\begin{theorem}\label{t_2}
 Let $A$ and $\delA$ be real $m \times n$ matrices, 
 with $\rank(A)=n$ and $\norm{\delA}_2\norm{A^\dagger}_2 \leq 1/2$. Then 
 $$\frac{|\ellDelta_j - \ell_j|}{\ell_j}  \leq 
4\>\left( \sqrt{\frac{1-\ell_j}{\ell_j}} + \frac{\kappa_2(A)}{\ell_j}\>\epsilon^{\perp}\right)\>
\kappa_2(A)\>\epsilon^{\perp},$$
and
 $$\frac{|\ellDelta_j - \ell_j|}{\ell_j}  \leq \left( 2\>\sqrt{\frac{1-\ell_j}{\ell_j}} + 
 \frac{\kappa_2(A)}{\ell_j}\>\epsilon\right)\>\kappa_2(A)\>\epsilon, \qquad 1\leq j\leq m.$$
\end{theorem}

\begin{proof}
See Section \ref{s_t2proof}.
\end{proof}

Theorem \ref{t_2} implies that relative leverage score differences
are bounded by the condition number $\kappa_2(A)$, a norm wise
perturbation, and a function that depends on the size of the leverage scores. 
Thus, an individual leverage score is well conditioned, if it has large magnitude and
if $A$ is well-conditioned with respect to left inversion.

The first bound in Theorem~\ref{t_2} recognizes, through the use of 
the projected perturbation $\epsilon^{\perp}$,
when the column spaces of $A$ and $A+\delA$ are close. In particular, $\epsilon^{\perp}=0$
for  $\range{A}=\range{A+\delA}$. The second bound does not do this,
but has the advantage of being simpler.  Although the first bound  
contains a smaller perturbation, $\epsilon^{\perp}\leq \epsilon$, it also has an additional factor
of~2. Therefore it is not clear that, in general, the first bound is tighter than the second one.

\subsubsection*{Numerical experiments: Figure~\ref{f_fig2}}
We demonstrate that both bounds capture the 
qualitative behavior of the leverage score sensitivities, but that the 
first bound appears more accurate when 
the perturbation has a substantial contribution in $\range{A}$.

\begin{figure}
\begin{center}
\includegraphics[width=5in,height=3.5in]{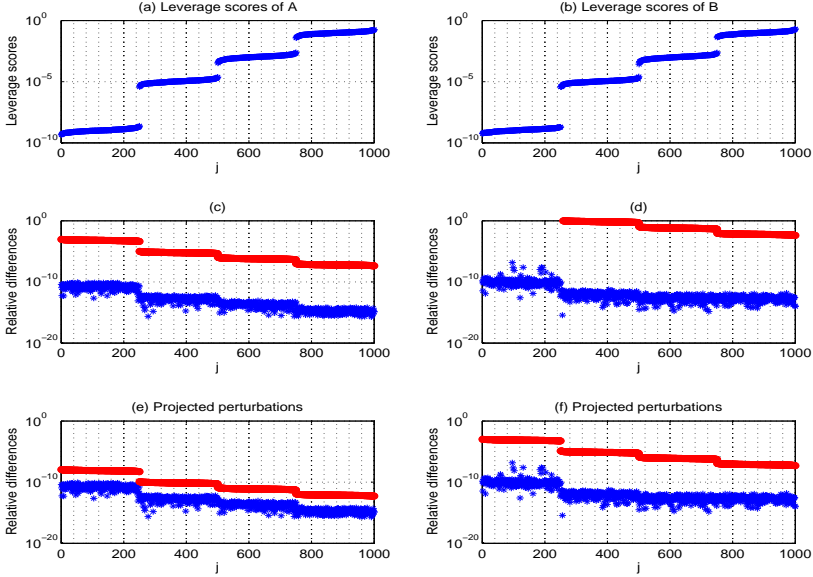}
\end{center}
\caption{(a)--(b) Leverage scores for matrices $A$ in \ref{e_A1} and $B$ in~\ref{e_B},
(c)--(f) relative leverage score differences $|\ellDelta_j-\ell_j|/\ell_j$ (blue stars) and 
bounds from Theorem~\ref{t_2} (red line above the stars) vs index~$j$, for 
perturbations $epsilon=10^{-8}$ in (c) and (d), and $epsilon^{\perp}\approx 10^{-14}$
in (e) and (f). }\label{f_fig2}
\end{figure}

Figure~\ref{f_fig2} shows the relative leverage score sensitivities and the bounds from
Theorem~\ref{t_2}, for perturbations $\epsilon=\|\delA\|_2/\|A\|_2$ and their projections
$\epsilon^{\perp}=\|(I-AA^{\dagger})\,\delA\|_2/\|A\|_2$.
Panels (a) and (b) in Figure~\ref{f_fig2} show the leverage scores  
for perfectly conditioned matrices $A$ constructed as in (\ref{e_A1}), and for matrices $B$
with $\kappa_2(B)\approx 10^{5}$ and leverage scores like those of $A$,
\begin{eqnarray}\label{e_B}
B=\diag\begin{pmatrix}I_{250} & 10^2\,I_{250} & 10^3\, I_{250} & 10^{4}\, I_{250}
 \end{pmatrix}\>\mathsf{gallery('randsvd', [m, n], 10^6 , 3)}.\quad
 \end{eqnarray}
Panels (c)--(f) in Figure~\ref{f_fig2} show the relative leverage score differences for $A$ and $B$
under two-norm perturbations $\epsilon=10^{-8}$.

Panels (c) and (e) show that the relative leverage score differences 
from the well-conditioned matrix $A$ 
reflect the leverage score distribution. That is, the smallest leverage scores
have relative differences of about $10^{-10}$, while the largest leverage scores have relative differences
that are several magnitude lower. 
The relative leverage score differences of the worse conditioned matrix $B$ in panels (d) and (f)
can be as high as $10^{-7}$, and do not follow the leverage distribution quite as clearly.
Therefore, the relative leverage score differences from norm wise perturbations increase with the
condition number of the matrix.

Panels (c) and (d) show the bound with $\epsilon$ in Theorem~\ref{t_2}, while 
(e) and (f) show the bound for the projected perturbation 
$\epsilon^{\perp}\approx 6\cdot 10^{-14}$ for both matrices.   Note that, as explained above, 
the bounds with the projected perturbations are not \textit{guaranteed} to be tighter,
although in this case they are several orders of magnitude more accurate.

\section{Leverage scores computed with  a QR decomposition}\label{s_qrpert}
We derive bounds for relative leverage score differences for leverage scores that are 
computed with a QR decomposition. 
The bounds assume exact arithmetic and  are based on perturbation results for QR decompositions;  they make no reference to particular QR implementations. 

Specifically, our bounds include: Norm-wise bounds for general matrix perturbations
(Section~\ref{s_norm}), bounds for general perturbations that recognize
row-scaling in the perturbations (Section~\ref{s_row}), and bounds for component-wise row-scaled
perturbations (Section~\ref{s_comp}). 
Since the bounds do not exploit the zero structure of the triangular
factor in the QR decomposition, they can be readily extended to  the polar decomposition as well. 

\subsection{General normwise perturbations}\label{s_norm}
The first bound is derived from a normwise perturbation result for QR 
decompositions \cite[Theorem 1.6]{Sun1991}.  Among the existing
and sometimes tighter QR perturbation bounds
\cite{BM1994,Chang2012,CP2001,CPS1997,CS2010,Stewart1977,Stewart1993,Sun1992,Sun1995,Zha1993}, we chose \cite[Theorem 1.6]{Sun1991} because it 
is simple and has the required key ingredients.

\begin{theorem}\label{t_4sun}
Let $A$ and $A+\delA$ be real $m \times n$ matrices with $\rank(A)=n$ and 
$\norm{\delA}_2\norm{A^\dagger} \leq 1/2$.
The leverage scores $\ellDelta_j$ computed from a QR decomposition of $A+\delA$ satisfy 
$$\frac{|\ellDelta_j - \ell_j|}{\ell_j}  \leq 12\>\left( \sqrt{\frac{1-\ell_j}{\ell_j}} + 
3 \,\frac{\kappa_2(A)\,\sr{A}^{1/2}}{\ell_j}\>\epsilon_F\right)\>\kappa_2(A)\>\sr{A}^{1/2}\>
\epsilon_F, \quad 1\leq j\leq m.$$
\end{theorem}

\begin{proof}
See Section \ref{s_t4sunproof}.
\end{proof}

The perturbation bound in Theorem~\ref{t_4sun} sends the message that: If $\delA$ is a
 general perturbation, then leverage scores computed from a QR
decomposition of $A+\delA$, are well-conditioned in the norm-wise relative sense, if they 
have large magnitude and if $A$ is well-conditioned. We demonstrate that this conclusion is valid in the following experiment. 

\subsubsection*{Numerical experiments: Figure~\ref{f_fig3}}
For matrices $A$ in (\ref{e_A1}), Figure~\ref{f_fig3} shows the relative leverage score 
differences $|\ellDelta_j-\ell_j|/\ell_j$ from norm-wise perturbations 
$\epsilon_F=\|\delA\|_F/\|A\|_F$ and the bound from Theorem~\ref{t_4sun},
for two different perturbations: $\epsilon_F=10^{-8}$ and $\epsilon_F=10^{-5}$.  

\begin{figure}
\begin{center}
\includegraphics[width=5in,height=3.5in]{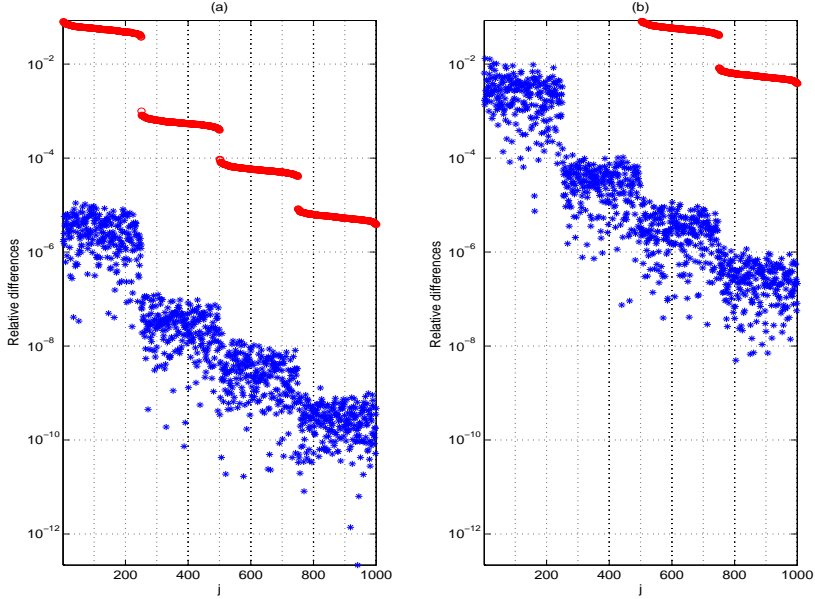}
\end{center}
\caption{Relative leverage score differences $|\ellDelta_j-\ell_j|/\ell_j$ (blue stars) and bound 
from Theorem~\ref{t_4sun} (red line above the stars)
vs index~$j$ for $\epsilon_F=10^{-8}$ (a) and $\epsilon_F=10^{-5}$ (b). 
}\label{f_fig3}
\end{figure}

Figure~\ref{f_fig3} illustrates that
the relative leverage score differences decrease with the same step size
with which the leverage score magnitude increases. In particular, for 
$\epsilon_F=10^{-8}$ in panel (a), the relative leverage score differences decrease from $10^{-5}$ 
for the smallest leverage scores to about $10^{-9}$ for the largest leverage scores. 
The differences for $\epsilon_F=10^{-5}$ in panel (b) are larger by a factor of 1000;
the 250 smallest leverage scores have lost all  accuracy because
they are smaller than the  perturbation $\epsilon_F$.

The bound in Theorem~\ref{t_4sun} differs from the actual differences by several
orders of magnitude, but reflects the qualitative behavior of the relative leverage score differences. 

\subsection{General normwise perturbation bounds that detect row scaling in the 
perturbations}\label{s_row}
The two first-order bounds presented here are based on a perturbation of the QR decomposition. 
Although the bounds make no assumptions on the perturbations $\delA$,
they are able to recognize row-scaling in $\delA$ of the form 
$$\epsilon_j\equiv \frac{\|e_j^T\delA\|_2}{\|e_j^TA\|_2}, \qquad 1\leq j\leq m.$$

\begin{theorem}\label{t_5}
Let $A$ and $A+\delA$ be real $m \times n$ matrices with $\rank(A)=n$  and 
$\norm{\delA}_2\norm{A^{\dagger}}_2 <1$.  The leverage scores $\ellDelta_j$ computed
from a QR decomposition of $A + \delA$ satisfy
\begin{equation*}
\frac{\left| \ellDelta_j - \ell_j \right|}{\ell_j} \leq 
2\>\left(\epsilon_j + \sqrt{2}\,\sr{A}^{1/2}\,\epsilon_F\right)\>\kappa_2(A) + 
\mathcal{O}(\epsilon_F^2), \qquad 1\leq j\leq m.
\end{equation*}
\end{theorem}

\begin{proof}
See Section~\ref{s_t5proof}.
\end{proof}

The relative leverage score difference bound for the $j$th leverage score in Theorem~\ref{t_5} contains three 
main ingredients: 
\begin{enumerate}
\item The two-norm condition number of $A$ with respect to left inversion, $\kappa_2(A)$.\\
It indicates leverage scores computed from matrices with smaller condition numbers have 
smaller relative leverage score differences.

\item The relative normwise perturbation in the $j$th row of $A$, $\epsilon_j$.\\
This perturbation represents the \textit{local} effect of $\delA$, because it shows
how the $j$th relative leverage score difference depends on the perturbation in row $j$ of~$A$.

\item The total normwise perturbation $\epsilon_F$.\\
This is the total relative mass of the perturbation, since 
$$\epsilon_F^2= \sum_{i=1}^m{\norm{e_i^T\delA}_2^2}/\norm{A}_F^2$$
represents the \textit{global} effect of $\delA$. 
\end{enumerate}

\subsubsection*{Numerical experiments: Figure~\ref{f_fig4}}
We illustrate that the local effect described above is real by examining the effect of row scaled perturbations on the relative accuracy of leverage 
scores computed with a QR decomposition.

\begin{figure}
\begin{center}
\includegraphics[width=5in,height=3.5in]{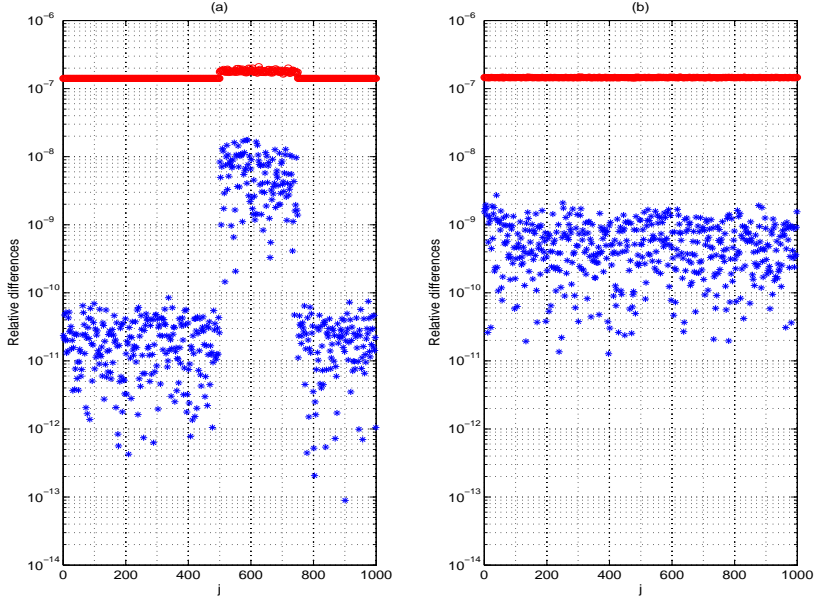}
\end{center}
\caption{Relative leverage score difference $|\ellDelta_j-\ell_j|/\ell_j$ (blue stars) and bound from Theorem~\ref{t_5}
(red line above the stars)
vs index~$j$ for row-wise scaled perturbations with $\epsilon_F=10^{-8}$.  In (a) only rows 501--750 of $A$ are perturbed, while in (b) the perturbation has the same row scaling 
as~$A$.}\label{f_fig4}
\end{figure}

Figure~\ref{f_fig4} shows the relative leverage score difference $|\ellDelta_j-\ell_j|/\ell_j$ from norm wise perturbations 
$\epsilon_F=\|\delA\|_F/\|A\|_F=10^{-8}$ and the bound from Theorem~\ref{t_5}. In
panel~(a), only rows 501--750 of $A$ are perturbed, while in panel~(b)
the perturbation has the same row scaling as $A$, that is,
$\delA=10^{-8} A1/\|A1\|_F$, where $A1$ is of the form (\ref{e_A1}).

In panel~(a), the leverage scores corresponding to rows 1--500 and 751-1000 have relative
 leverage score differences between $10^{-12}$ and~$10^{-10}$, which illustrates that the local 
 perturbation in rows 501--750 has
 a global effect on all leverage scores. However, the leverage scores corresponding to the
 perturbed rows 501--750 have larger relative differences of $10^{-8}$ or more, which illustrates the 
 strong effect of local perturbations. The bound from Theorem~\ref{t_5} hovers around $10^{-7}$, 
 but is slightly larger for the leverage scores corresponding to the perturbed rows. Thus, 
Theorem~\ref{t_5} is able to detect strongly local row scaling in norm wise perturbations.
 
In panel~(b), almost all leverage scores have relative differences between $10^{-10}$ and $10^{-9}$, and the bound from Theorem~\ref{t_5} is flat at~$10^{-7}$. Thus, the relative leverage
scores differences tend to be more uniform when the norm wise perturbations have the same
row scaling as the matrix.  This effect is recognized by Theorem~\ref{t_5}.

Therefore,  although Theorem~\ref{t_5} makes no assumptions about the 
perturbations $\delA$, it is able to detect row scaling in norm wise perturbations,
and correctly predicts the qualitative behavior of relative leverage score differences.

\subsubsection*{Projected perturbations}
The following bound is a refinement of Theorem \ref{t_5} that projects out the part of the perturbation that 
lies in $\range{A}$ and does not contribute to a change in leverage scores,
$$\epsilon^{\perp}_F\equiv \frac{\|(I_m-AA^{\dagger})\,\delA\|_F}{\|A\|_F},\qquad
\epsilon^{\perp}_j\equiv \frac{\|e_j^T(I-AA^{\dagger})\,\delA\|_2}{\|e_j^TA\|_2}, \qquad 1\leq j\leq m.$$

\begin{theorem}[Projected perturbations]\label{t_6}
Let $A$ and $A+\delA$ be real $m \times n$ matrices with $\rank(A)=n$ and
$\norm{\delA}_2\norm{A^\dagger}_2 \leq 1/2$.  The leverage scores $\ellDelta_j$ computed from
a QR decomposition of $A+\delA$ satisfy
\begin{equation*}
\frac{\left| \ellDelta_j - \ell_j \right|}{\ell_j} \leq 4\>\left( \epsilon^{\perp}_j
+ \sqrt{2}\,\sr{A}^{1/2}\,\epsilon^{\perp}_F \right)\kappa_2(A) + 
\mathcal{O}\left((\epsilon^{\perp}_F)^2\right), \qquad 1\leq j\leq m.
\end{equation*}
\end{theorem}

\begin{proof}
See Section \ref{s_t6proof}.
\end{proof}

It is not clear that Theorem~\ref{t_6} is  tighter than Theorem~\ref{t_5}. First, Theorem~\ref{t_6}
contains an additional factor of~2 in the bound. Second,
although the total projected perturbation is smaller, i.e. $\epsilon^{\perp}_F\leq \epsilon_F$,
this is not necessarily true for $\epsilon^{\perp}_j$ and $\epsilon_j$.
For instance, if
$$A = \tfrac{1}{2}\>\begin{pmatrix} 1 & 1 \\ 1 & -1 \\ 1 & 1 \\ 1 & -1 \end{pmatrix}, \qquad
\delA = \begin{pmatrix} 1 & 1 \\ 0 & 0 \\ 0 & 0 \\ 0 & 0 \end{pmatrix},$$
then 
$$(I-AA^{\dagger})\,\delA= (I-AA^T)\,\delA = 
\tfrac{1}{2}\>\begin{pmatrix} 1 & 1 \\ 0 & 0 \\ 1 & 1 \\ 0 & 0 \end{pmatrix}.$$
Here we have
$\epsilon_3=\norm{e_3^T\delA}_2 /\norm{e_3^TA}_2= 0$ and 
$\epsilon^{\perp}_3=\norm{e_3^T(I-AA^{\dagger})\,\delA}_2/\norm{e_3^TA}_2 = 1$, so that
$\epsilon^{\perp}_3>\epsilon_3$.

\subsection{Componentwise row-scaled perturbations}\label{s_comp}
Motivated by Section~\ref{s_row}, where bounds for general
perturbations $\delA$ can recognize row scaling in $\delA$., we ask the natural 
follow-up question:
What if $\delA$ does indeed represent a row scaling of $A$? Can we get tighter bounds?
To this end, we consider componentwise row perturbations of the form  $|e_j^T \delA|\leq \eta_j\> |e_j^TA|$,
where $\eta_j\geq 0$, $1\leq j\leq m$, and model them as 
\begin{eqnarray}\label{e_pertrow}
e_j^T\delA  = \zeta_j\, \eta_j\>e_j^TA, \qquad 1\leq j\leq m, \qquad
\eta \equiv  \max_{1\leq j\leq m}{\eta_j},
\end{eqnarray} 
where $\zeta_j$ are uniform random variables in $[-1, 1]$, $1\leq j\leq m$.
We show that, under component wise row-scaled perturbations (\ref{e_pertrow}),
leverage scores computed with a QR decomposition have
relative differences that do not depend, to first order, on the condition number or
the magnitudes of the leverage scores. 

\begin{theorem}\label{t_8}
Let $A$ be a real $m\times n$ matrix with $\rank(A)=n$, and let the perturbations $\delA$ be
of the form (\ref{e_pertrow}) with $\eta \,\kappa_2(A)< 1$. The leverage scores $\ellDelta_j$ 
computed from a QR decomposition of $A+\delA$ satisfy
\begin{eqnarray*}
\frac{\left| \ellDelta_j - \ell_j \right|}{\ell_j} \leq 2\>\left( \eta_j
+ \sqrt{2}\,n\>\eta\right) + \mathcal{O}(\eta^2), \qquad 1\leq j\leq m.
\end{eqnarray*}
\end{theorem}

\begin{proof}
See Section~\ref{s_t8proof}.
\end{proof}

The quantities $\eta_j$ represent the local effects of the individual row-wise perturbations, while 
the factor $n\,\eta$ represents the global effect of all perturbations. 
 In contrast to our previous results,  the bound does not depend on either the condition number 
 or the leverage score magnitude.

\subsubsection*{Numerical experiments: Figure~\ref{f_fig5}}
We illustrate the effect of component-wise row-scaled perturbations on the relative accuracy 
of leverage scores that are computed with a QR decomposition.

\begin{figure}
\begin{center}
\includegraphics[width=5in,height=3.5in]{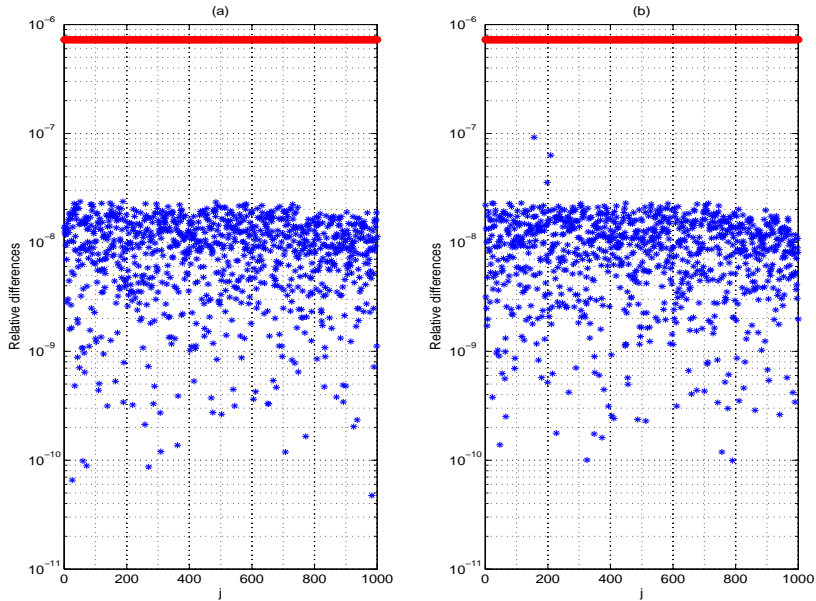}
\end{center}
\caption{Relative leverage score differences $|\ellDelta_j-\ell_j|/\ell_j$ (blue stars) and 
the bound from Theorem~\ref{t_6} (red line above the stars)
vs index~$j$ for component wise row-wise scaled perturbations with $\eta_j=10^{-8}$, 
$1\leq j\leq m$.}\label{f_fig5}
\end{figure}

Figure~\ref{f_fig5} shows the relative leverage score differences from
 a well-conditioned matrix $A$
with $\kappa_2(A)=1$ in (a), and from a worse conditioned matrix $B$ in (\ref{e_B}) 
with $\kappa_2(B)\approx 10^5$ in (b). The component-wise row-scaled
perturbations from (\ref{e_pertrow}) are $\eta=\eta_j=10^{-8}$ for $1\leq j\leq m$.
The leverage scores for these types of matrices are shown in Figure~\ref{f_fig2}. 

Figure~\ref{f_fig5} shows that the relative leverage score differences for both matrices look 
almost the  same, hovering around $10^{-8}$, except for a few outliers.
Thus, the relative accuracy of most leverage scores does not depend on the condition number,
but a few small leverage scores do show a slight effect. 
Note that Theorem~\ref{t_6} is based only on a perturbation analysis, not a round off error analysis
of the QR decomposition, and that we did not take into errors arising in the  
computation of the two norm.

Furthermore, Figure~\ref{f_fig5} shows that the  relative leverage score differences do not depend on the leverage score magnitude. Hence Theorem~\ref{t_6} captures the relative 
 leverage score accuracy under  component-wise row-scaled perturbations.

\section{Summary}\label{s_sum}
We analyzed the conditioning of individual leverage scores (Section~\ref{s_angles}) and took the first
steps in assessing the numerical stability of QR decompositions for computing
leverage scores (Section~\ref{s_qrpert}).   To this end,
we derived several bounds for the relative accuracy of individual leverage scores.
The bounds are expressed  in terms of principal angles between column spaces, and
three classes  of matrix perturbations:
General norm-wise, norm-wise row-scaled, and component-wise row-scaled.

Since most of the bounds in Section~\ref{s_qrpert} do not exploit the zero structure of 
the upper triangular factor, they are readily extended to polar decompositions as well.

\subsubsection*{Future research}
The next step is to extend the results in Section~\ref{s_comp} to component-wise perturbations 
$|\delA_{jk}|\leq \eta_{jk}|A_{jk}|$, $1\leq j\leq m$, $1\leq k\leq n$. Numerical experiments 
strongly suggest
that leverage scores computed from QR decompositions of such perturbed matrices 
have relative leverage score differences that do not depend on the magnitude of the leverage scores.

The most popular method for computing leverage scores is the singular value decomposition.
The numerical stability of the SVD in this context needs to be investigated, 
and whether the sensitivity of the singular vectors to singular value gaps matters for
leverage score computations.

Another issue is the numerically stable computation of "$k$-leverage scores". These are leverage scores of the best rank~$k$ approximation to $A$ in the two-norm. Determining leverage scores 
from a truncated SVD is necessary when $A$ is (numerically) rank deficient, or when noisy data
are well represented, as in the case of PCA, by only a few dominant singular vectors.

\appendix
\section{Proofs}\label{s_proofs}
We present proofs for the results in Sections \ref{s_angles} and~\ref{s_qrpert}.

\subsection{Proof of Theorem~\ref{t_1}}\label{s_t1proof}
The full column rank of $A$ and $A+\delA$ assures that the leverage scores are
well-defined according to Definition~\ref{d_lev}.

The proof proceeds in three stages: Expressing the perturbed leverage scores $\ellDelta_j$
in terms of the exact leverage scores $\ell_j$; an upper bound for $\ellDelta_j-\ell_j$; and a
lower bound for $\ellDelta_j-\ell_j$.

\paragraph{1. Expressing the perturbed leverage scores in terms of the exact ones}
Consider any orthonormal basis for the column spaces:
Let $A=QX$, where $X$  is nonsingular and $Q^TQ=I_n$.
Similarly, let $A+\delA=\delQ\delX$, where $\delX$ is nonsingular and $\delQ^T\delQ=I_n$. The leverage scores are $\ell_j=\|e_j^TQ\|_2^2$ and $\ellDelta_j=\|e_j^T\delQ\|_2^2$, $1\leq j\leq m$.

With Definition~\ref{d_angles}, rotate to the basis of principal vectors
$Q_1\equiv QU$ and $\delQ_1\equiv \delQ V$, which satisfy $Q_1^T\delQ_1=\Sigma$. Since the
leverage scores are basis independent we can write $\ell_j=\|e_j^TQ_1\|_2^2$ and
$\ellDelta_j=\|e_j^T\delQ_1\|_2^2$, $1\leq j\leq m$.

The goal is to express $\delQ_1$ in terms of $Q_1$. To this end
choose $Q_2$ so that $\mathcal{Q}\equiv \begin{pmatrix}Q_1 &Q_2\end{pmatrix}$ is a $m\times m$
orthogonal matrix. Then $I_m=\mathcal{Q}\mathcal{Q}^T=Q_1Q_1^T+Q_2Q_2^T$ implies
\begin{eqnarray}\label{e_t1a}
\|e_j^TQ_2\|_2^2=e_j^TQ_2Q_2^Te_j=e_j^Te_j-e_j^TQ_1Q_1^Te_j=1-\|e_j^TQ_1\|_2^2=1-\ell_j.
\end{eqnarray}
Write
$$\mathcal{Q}^T\delQ_1=\begin{pmatrix} Q_1^T\delQ_1\\ Q_2^T\delQ_1\end{pmatrix}
=\begin{pmatrix}\Sigma \\ Z\end{pmatrix},\qquad where \quad  Z\equiv Q_2^T\delQ_1.$$
This gives the desired expressions $\delQ_1=Q_1\Sigma +Q_2Z$ and
$$\delQ_1\delQ_1^T=
Q_1\Sigma^2Q_1^T+Q_2Z\Sigma Q_1^T+Q_1\Sigma Z^TQ_2^T +Q_2ZZ^TQ_2^T.$$  
From $\mathcal{Q}^T\delQ_1$ having orthonormal columns follows $Z^TZ=I_n-\Sigma^2$
and 
\begin{eqnarray}\label{e_t1b}
\|Z\|_2=\sqrt{\|I_n-\Sigma^2\|_2}=\sqrt{1-(\cos{\theta_n})^2}=\sin{\theta_n}.
\end{eqnarray}
At last, we can express the perturbed leverage scores in terms of the exact ones,
\begin{eqnarray}\label{e_t1c}
\ellDelta_j &=& e_j^T\delQ_1\delQ_1^Te_j= e_j^TQ_1\>\Sigma^2\>Q_1^Te_j+2 
e_j^TQ_2\>Z\Sigma\> Q_1^Te_j+e_j^TQ_2\>ZZ^T\>Q_2^Te_j
\end{eqnarray}

\paragraph{2. Upper bound}
Applying (\ref{e_t1a}) and (\ref{e_t1b}) in (\ref{e_t1c}) gives
\begin{eqnarray}
\ellDelta_j &\leq&
\|\Sigma\|_2^2\>\ell_j+ 2 \|e_j^TQ_2\|\>\|Z\|_2\|\Sigma\|_2 \>\sqrt{\ell_j}+\|e_j^TQ_2\|_2^2 \|Z\|_2^2 \notag\\
&\leq& (\cos{\theta_1})^2\ell_j+2\cos{\theta_1}\sin{\theta_n}\>\sqrt{\ell_j(1-\ell_j)}+
(\sin{\theta_n})^2\>(1-\ell_j),\label{e_t1d}\\
&=& \left((\cos{\theta_1})^2-(\sin{\theta_n})^2\right)\>\ell_j
+2\cos{\theta_1}\sin{\theta_n}\>\sqrt{\ell_j(1-\ell_j)} +(\sin{\theta_n})^2.\notag
\end{eqnarray}
Subtracting $\ell_j$ on both sides, and omitting the summand with the negative sign gives 
\begin{eqnarray*}
\ellDelta_j-\ell_j &\leq& 
-\left((\sin{\theta_1})^2+(\sin{\theta_n})^2\right)\>\ell_j
+2\cos{\theta_1}\sin{\theta_n}\>\sqrt{\ell_j(1-\ell_j)} +(\sin{\theta_n})^2\\
&\leq &2\cos{\theta_1}\sin{\theta_n}\>\sqrt{\ell_j(1-\ell_j)} +(\sin{\theta_n})^2.
\end{eqnarray*}
If $m=2n$, write (\ref{e_t1d}) as
$$\ellDelta_j\leq \left(\cos{\theta_1}\sqrt{\ell_j}+ \sin{\theta_j}\sqrt{1-\ell_j}\right)^2.$$

\paragraph{3. Lower bound}
Applying (\ref{e_t1a}) and (\ref{e_t1b}) in (\ref{e_t1c}) gives
\begin{eqnarray*}
\ellDelta_j &\geq&
(\cos{\theta_n})^2\>\ell_j- 2 \|e_j^TQ_2\|\>\|Z\|_2\|\Sigma\|_2 \>\sqrt{\ell_j}\\
&=&
(\cos{\theta_n})^2\>\ell_j -2 \>\cos{\theta_1}\sin{\theta_n} \sqrt{\ell_j(1-\ell_j)}.
\end{eqnarray*}
Subtracting $\ell_j$ on both sides gives, and using $\ell_j\leq 1$ gives
\begin{eqnarray*}
\ellDelta_j -\ell_j&\geq&
-(\sin{\theta_n})^2\>\ell_j -2 \>\cos{\theta_1}\sin{\theta_n} \sqrt{\ell_j(1-\ell_j)}\\
&\geq &-(\sin{\theta_n})^2 -2 \>\cos{\theta_1}\sin{\theta_n} \sqrt{\ell_j(1-\ell_j)}.
\end{eqnarray*}

If $m=2n$ then $Z$ is a square matrix, and the smallest eigenvalue of $ZZ^T$ is equal to
$$\lambda_n(ZZ^T)=\lambda_n(Z^TZ)=\lambda_n(I_n-\Sigma^2)=1-(\cos{\theta_1})^2.$$
Applying this in (\ref{e_t1c}) gives
\begin{eqnarray*}
\ellDelta_j &\geq&
(\cos{\theta_n})^2\>\ell_j- 2 \|e_j^TQ_2\|\>\|Z\|_2\|\Sigma\|_2 \>\sqrt{\ell_j}+\left(1-(\cos{\theta_1})^2\right)
\>\|e_j^TQ_2\|_2^2\\
&=& (\cos{\theta_n})^2\>\ell_j -2 \>\cos{\theta_1}\sin{\theta_n} \sqrt{\ell_j(1-\ell_j)}
+\left(1-(\cos{\theta_1})^2\right)\>(1-\ell_j)\\
&=& 1-\left(\sin{\theta_n}\>\sqrt{\ell_j}+\cos{\theta_1}\>\sqrt {1-\ell_j}\right)^2.
\end{eqnarray*}

\subsection{Proof of Theorem~\ref{t_2}}\label{s_t2proof}
The assumption $\|\delA\|_2\>\|A^{\dagger}\|_2<1$ implies that $\rank(A+\delA)=\rank(A)=n$, hence 
the perturbed leverages scores $\ellDelta_j$ are well-defined.

We express $\sin{\theta_n}$ in terms of orthogonal projectors onto the column spaces.
Let $\mathcal{P}$ be the orthogonal projector onto $\range{A}$, and $\tilde{\mathcal{P}}$ the orthogonal
projector onto $\range{A+\delA}$. Then  \cite[Theorem 5.5]{StS90}, \cite[(5.6)]{wedin72} implies
$$\sin{\theta_n}=\|(I_m-\mathcal{P})\tilde{\mathcal{P}}\|_2.$$
From $\mathcal{P}=AA^{\dagger}$ and $\tilde{\mathcal{P}}=(A+\delA)(A+\delA)^{\dagger}$ follows
\begin{eqnarray*}
\sin{\theta_n}&=&\|(I_m-AA^{\dagger})\>(A+\delA)(A+\delA)^{\dagger}\|_2=
\|(I_m-AA^{\dagger})\>\delA\>(A+\delA)^{\dagger}\|_2\\
&\leq& \|(I_m-AA^{\dagger})\>\delA\|_2 \|(A+\delA)^{\dagger}\|_2 =
\|A\|_2\,\|(A+\delA)^{\perp}\|_2\>\epsilon^{\perp}.
\end{eqnarray*}
It remains  to express $\|(A+\delA)^{\dagger}\|_2$ in terms of $\|A^{\dagger}\|_2$. The 
well-conditioning of singular values \cite[Corollary 8.6.2]{GovL13} implies 
$$\|(A+\delA)^{\dagger}\|_2\leq 
\frac{\|A^{\dagger}\|_2}{1-\|\delA\|_2\|A^{\dagger}\|_2}\leq 2\>\|A^{\dagger}\|_2,$$
where the last inequality is due to the assumption
$\|\delA\|_2\|A^{\dagger}\|_2\leq 1/2$.
Substituting this into the bound for $\sin{\theta_n}$ yields
$\sin{\theta_n}\leq 2\> \kappa_2(A)\>\epsilon^{\perp}$.
In turn now, inserting this into Corollary~\ref{c_1}, and bounding 
$\cos{\theta_1}\leq 1$ gives the first bound in Theorem~\ref{t_2}.
The second one follows more easily from $\sin{\theta_n} \leq \kappa_2(A)\>\epsilon$,
see \cite[(4.4)]{wedin83}.

\subsection{Proof of Theorem \ref{t_4sun}}\label{s_t4sunproof}
We start with a special case of Theorem~\ref{t_2} 
applied to $m\times n$ matrices $Q$ and $Q+\Delta Q$ with orthonormal
columns and leverage scores $\ell_j=\|e_j^TQ\|_2^2$ and
$\ellDelta_j=\|e_j^T(Q+\Delta Q)\|_2^2$. Since $\|Q\|_2=\kappa_2(Q)=1$, we obtain
\begin{eqnarray}\label{e_pertq4}
\frac{|\ellDelta_j - \ell_j|}{\ell_j}  \leq \left( 2\>\sqrt{\frac{1-\ell_j}{\ell_j}} + 
 \frac{\|\Delta Q\|_2}{\ell_j}\right)\>\|\Delta Q\|_2,\qquad 1\leq j\leq m.
 \end{eqnarray}

The bound for $\|\Delta Q\|_2\leq \|\Delta Q\|_F$ is obtained from 
a simpler version of the  lemma below. 

\begin{lemma}[Theorem 1.6 in \cite{Sun1991}]\label{l_4sun}
Let $A$ and $\delA$ be real $m \times n$ matrices with $\rank(A)=n$, and
$\norm{A^\dagger}_2\norm{\delA}_2 < 1$.  If $A + \delA = (Q + \Delta Q)\,\tilde{R}$ is the 
thin QR decomposition, then
$$\norm{\Delta Q}_F \leq \frac{1+\sqrt{2}}{1-\norm{A^\dagger}_2\norm{\delA}_2}\>\|A^{\dagger}\|_2
\norm{\delA}_F.$$
\end{lemma}

Below  is a simpler but not much more restrictive version of Lemma~\ref{l_4sun}.
If $\|\delA\|_2\|A^{\dagger}\|_2\leq 1/2$, then
\begin{eqnarray*}
\norm{\Delta Q}_F \leq 6\>\|A^{\dagger}\|_2\norm{\delA}_F = 6 \,\sr{A}^{1/2}\,\kappa_2(A)\>\epsilon_F.
\end{eqnarray*}
Substituting this into (\ref{e_pertq4}) gives Theorem~\ref{t_4sun}.

\subsection{Proof of Theorem \ref{t_5}}\label{s_t5proof}
We start with a simplified version of Theorem~\ref{t_2}.
Let $Q$ and $Q+\Delta Q$ be $m\times n$ matrices with orthonormal columns, and
$\ell_j=\|e_j^TQ\|_2^2$ and $\ellDelta_j=\|e_j^T(Q+\Delta Q)\|_2^2$, $1\leq j\leq m$, their leverage
scores. Multiplying out the inner product in $\ellDelta_j$ and using triangle and submultiplicative
inequalities gives
\begin{eqnarray}\label{e_pertq}
\frac{|\ellDelta_j-\ell_j|}{\ell_j} \leq 2 \> \frac{\|e_j^T\Delta Q\|_2}{\sqrt{\ell_j}}+
\frac{\|e_j^T\Delta Q\|_2^2}{\ell_j}, \qquad 1\leq j\leq m.
\end{eqnarray}
Next we derive bounds for $\|e_j^T\Delta Q\|_2$ in terms of $\delA$.
To this end we represent the perturbed matrix by a function $A(t)$, 
with a smooth decomposition $A(t) = Q(t)R(t)$. 

This is a very common approach, see for instance
\cite[Section 3]{Chang2012},
\cite[Section 4]{CP2001},
\cite[Section 3]{CPS1997},
\cite[Section 5]{CS2010},
\cite[Section 2.1]{DE1999}
\cite[Section 2.4]{Higham2002},
\cite[Section 3]{Stewart1977}, 
\cite[Section 2]{Sun1991}, \cite[Section 4]{Sun1992}, \cite[Section 5]{Sun1995}, and
\cite[Section]{Zha1993}.

Define the function
$$ A(t)\equiv A+\frac{t}{\epsilon_F}\>\delA, \qquad 
0\leq t\leq\epsilon_F\equiv\frac{\|\delA\|_F}{\|A\|_F}.$$
Let $A(t)=Q(t)R(t)$ be a thin QR decomposition, where we set
$Q\equiv Q(0)$, $R\equiv R(0)$, $Q+\Delta Q\equiv Q(\epsilon_F)$ and 
$R+\Delta R \equiv R(\epsilon_F)$.
The derivative of $R$ with regard to $t$ is $\dot{R}$. 

\begin{theorem}\label{t_7}
Let $A$ and $A+\delA$ be real $m \times n$ matrices with $\rank(A)=n$ and 
$\|\delA\|_2\|A^{\dagger}\|_2<1$.  Then 
$$\Delta Q =\delA\,R^{-1} - \epsilon_F\>Q\,\dot{R}\,R^{-1}+\mathcal{O}(\epsilon_F^2), $$
where 
$\|\dot{R}\,R^{-1}\|_F\leq \sqrt{2}\, \sr{A}^{1/2} \,\kappa_2(A)$.
 \end{theorem}
\smallskip

\begin{proof}
The proof is inspired in particular by \cite[Section~4]{CP2001} and \cite[Section 2.4]{Higham1986}.

\paragraph{Smooth decomposition}
From $\rank(A)=n$, $\|\tfrac{t}{\epsilon_F}\,\delA\|_2\leq \|\delA\|_2$ for $0\leq t\leq \epsilon_F$, and
$\|\delA\|_2\|A^{\dagger}\|_2<1$ follows $\rank(A(t))=n$.
Furthermore, since $A(t)$ has at least two continuous derivatives,
so do $Q(t)$ and $R(t)$ \cite[Proposition 2.3]{DE1999}. 

\paragraph{Expression for $\Delta Q$}
The existence of two derivatives allows us to take a Taylor expansion of $Q(t)$ around $t=0$, and get
$Q(t)-Q(0)=t\,\dot{Q}(0) +\mathcal{O}(t^2)$. Evaluating at $t=\epsilon_F$ gives
\begin{eqnarray}\label{e_t7a}
\Delta Q =(Q+\Delta Q)-Q=Q(\epsilon_F)-Q(0)=\epsilon_F\, \dot{Q} + \mathcal{O}(\epsilon_F^2).
\end{eqnarray}
To get an expression for $\dot{Q}$, differentiate $A(t)=Q(t)R(t)$,
$$\frac{\delA}{\epsilon_F} =\dot{Q}(t)\,R(t) +Q(t)\,\dot{R}(t),$$
and evaluate at $t=0$,
$$\dot{Q} = \frac{\delA}{\epsilon_F}\,R^{-1}  - Q\,\dot{R}\,R^{-1}.$$
Inserting the above into (\ref{e_t7a}) gives the expression for $\Delta Q$ in Theorem~\ref{t_5}.
\smallskip

\paragraph{Bound for $\|\dot{R}R^{-1}\|_F$}
Differentiating $A(t)^TA(t)=R(t)^TR(t)$ gives 
$$\frac{1}{\epsilon_F}\>\left( (\delA)^TA +  A^T\delA +\frac{2t}{\epsilon}\,
(\delA)^T\delA\right) = \dot{R}(t)^TR(t)+R(t)^T\dot{R}(t),$$
and evaluating at $t=0$ yields
\begin{eqnarray*}
\frac{1}{\epsilon_F}\> \left( (\delA)^TA +  A^T\delA\right) = \dot{R}^TR+R^T\dot{R}.
\end{eqnarray*} 
Multiplying by $R^{-T}$ on the left and by $R^{-1}$ on the right gives
\begin{equation}\label{e_t7c}
\dot{R}R^{-1} + \left( \dot{R}R^{-1} \right)^T = \frac{1}{\epsilon_F}\> \left(Q^T\delA R^{-1} + 
\left( Q^T\delA R^{-1} \right)^T\right).
\end{equation}
Now we take advantage of the fact that $\dot{R}R^{-1}$ is upper triangular, and define 
a function that extracts the upper triangular part of a square matrix $Z$ via
$$\myup(Z) \equiv \frac{1}{2}\mathrm{diagonal}(Z)+\mathrm{strictly~upper~triangular~part}(Z).$$
Applying the function to (\ref{e_t7c}) gives 
$$\dot{R}R^{-1}  = \frac{1}{\epsilon_F}\>\myup\left( Q^T\delA R^{-1} + 
\left( Q^T\delA R^{-1} \right)^T\right).$$ 
Taking norms yields \cite[Equation (3.5)]{CP2001}
\begin{eqnarray}\label{e_t7d}
\norm{\dot{R}R^{-1}}_F &\leq &\frac{\sqrt{2}}{\epsilon_F}\>\norm{Q^T\delA  R^{-1}}_F\label{e_rr}\\
&\leq&\frac{\sqrt{2}}{\epsilon_F}\>\norm{\delA}_F  \|R^{-1}\|_2
=\sqrt{2}\>\sr{A}^{1/2}\,\kappa_2(A).\notag
\end{eqnarray}
\end{proof}

Now we are ready to derive a bound for the row norms of $\Delta Q$. Combining the
two bounds from Theorem~\ref{t_7}, that is, inserting
$\|\dot{R}\,R^{-1}\|_F\leq \sqrt{2}\, \sr{A}^{1/2} \,\kappa_2(A)$ into
$$\norm{e_j^T\Delta Q}_2 \leq \norm{e_j^T\delA}_2\norm{A^{\dagger}}_2 + \epsilon_F\,\sqrt{\ell_j}\norm{\dot{R}R^{-1}}_2 + \mathcal{O}(\epsilon^2), \qquad 1\leq j\leq m,$$
gives 
$$\norm{e_j^T\Delta Q}_2 \leq \norm{e_j^T\delA}_2 \|A^{\dagger}\|_2 + 
\sqrt{2\,\ell_j}\,\sqrt{\sr{A}}\>\epsilon_F \>\kappa_2(A) + \mathcal{O}(\epsilon_F^2).$$
Into the first summand substitute
\begin{eqnarray*}
\|e_j^T\delA\|_2=\epsilon_j\,\|e_j^TA\|_2\leq \epsilon_j\,\|e_j^TQ\|_2\,\|R\|_2=
\epsilon_j\,\sqrt{\ell_j}\,\|A\|_2,
\end{eqnarray*}
and obtain
$$\|e_j^T\Delta Q\|_2
\leq\sqrt{\ell_j}\>\left(\epsilon_j + \sqrt{2}\,\sqrt{\sr{A}}\>\epsilon_F \right)\>\kappa_2(A) + 
\mathcal{O}(\epsilon_F^2), \qquad 1\leq j\leq m.$$
Inserting the above into (\ref{e_pertq}) and focussing on the first order terms in $\epsilon_F$
gives Theorem~\ref{t_5}.

\subsection{Proof of Theorem~\ref{t_6}}\label{s_t6proof}
To remove the contribution of $\delA$ in $\range{A}$,
let $\mathcal{P}\equiv AA^{\dagger}$ be the orthogonal projector
onto $\range{A}$, and $\mathcal{P}^{\perp}\equiv I_m-\mathcal{P}$ the orthogonal projector
onto $\range{A}^{\perp}$. Extracting the contribution in $\range{A}$ gives
$$A + \delA = A + \mathcal{P}\,\delA + \mathcal{P}^{\perp}\,\delA 
=\left(A+ \mathcal{P}\,\delA\right) +\mathcal{P}^{\perp}\,\delA = M +\Delta M,$$
where $M\equiv A+\mathcal{P}\,\delA$ and $\Delta M\equiv \mathcal{P}^{\perp}\>\delA$.

\paragraph{Leverage scores}
Here $\rank(M)=n$, because $\mathcal{P}$ is an orthogonal projector, so that
$\|\mathcal{P}\,\delA\|_2\|A^{\dagger}\|_2\leq \|\delA\|_2\|A^{\dagger}\|_2<1$.
With $M=\mathcal{P}\>(A+\delA)$ this implies $\range{M}=\range{A}$.
Furthermore $\rank(M+\Delta M)=\rank(A+\delA)=n$.
Thus $M$ and $M+\Delta M$ have thin QR decompositions $M=QX$ and 
$M+\Delta M=(Q+\Delta Q)\tilde{X}$, and 
have the same leverage scores $\ell_j$ and $\ellDelta_j$, respectively, as $A$ and
$A+\delA$.

Ultimately, we want to apply Theorem~\ref{t_5}
to $M$ and $M+\Delta M$, but the perturbation
$\Delta M=\mathcal{P}^{\perp}\,\delA$ is to be related to $A$ rather than to $M$,
and the bound is to be expressed in terms of $\kappa_2(A)$ rather than $\kappa_2(M)$.

\paragraph{Applying Theorem~\ref{t_7}} 
With 
$$M(t)\equiv M+\tfrac{t}{\mu}\Delta M, \qquad 0\leq t\leq 
\mu\equiv \frac{\|\Delta M\|_F}{\|A\|_F}=\epsilon^{\perp}_F,$$
Theorem~\ref{t_7} implies 
\begin{eqnarray}\label{e_t7e}
\Delta Q =\Delta M\,X^{-1} - \mu\>Q\,\dot{X}\,X^{-1}+\mathcal{O}(\mu^2).
\end{eqnarray}
To bound $\|\dot{X}X^{-1}\|_F$, we apply (\ref{e_t7d}) and obtain
\begin{eqnarray}\label{e_t7f}
\|\dot{X}\,X^{-1}\|_F\leq \frac{\sqrt{2}}{\mu}\>\|\Delta M\|_F\,\|X^{-1}\|_2.
\end{eqnarray}
\paragraph{Bounding $\|e_j^T\Delta Q\|_2$}
Combining (\ref{e_t7e}) and (\ref{e_t7f}) gives
\begin{eqnarray*}
\|e_j^T\Delta Q\|_2&\leq& \left(\|e_j^T\Delta M\|_2+\sqrt{2}\,\|e_j^TQ\|_2\,\|\Delta M\|_F\right)\>
\|X^{-1}\|_2+\mathcal{O}(\mu^2), \qquad 1\leq j\leq m\\
&=& \left(\epsilon^{\perp}_j\,\frac{\|e_j^TA\|_2}{\|A\|_2} +
\sqrt{2}\,\sqrt{\ell_j}\,\mu\,\sr{A}^{1/2}\right)\>\|A\|_2\|M^{\dagger}\|_2+\mathcal{O}(\mu^2).
\end{eqnarray*}
From $\|e_j^TA\|_2\leq \|e_j^TQ\|_2\|A\|_2=\sqrt{\ell_j}\,\|A\|_2$ follows
\begin{eqnarray}\label{e_t7g}
\|e_j^T\Delta Q\|_2&\leq & \sqrt{\ell_j}\>\left(\epsilon^{\perp}_j +
\sqrt{2}\,\mu\,\sr{A}^{1/2}\right)\>\|A\|_2\|M^{\dagger}\|_2+\mathcal{O}(\mu^2).
\end{eqnarray}
It remains  to express $\|M^{\dagger}\|_2$ in terms of $\|A^{\dagger}\|_2$. The 
well-conditioning of singular values \cite[Corollary 8.6.2]{GovL13} applied to 
$M=A+Z$, where $Z\equiv \mathcal{P}\,\delA$, implies 
$$\|M^{\dagger}\|_2=\|(A+Z)^{\dagger}\|_2\leq 
\frac{\|A^{\dagger}\|_2}{1-\|Z\|_2\|A^{\dagger}\|_2}\leq 2\>\|A^{\dagger}\|_2,$$
where the last inequality is due to the assumption
$\|Z\|_2\|A^{\dagger}\|_2\leq \|\delA\|_2\|A^{\dagger}\|_2\leq 1/2$.
Inserting this bound for $\|M^{\dagger}\|_2$ into (\ref{e_t7g}) yields
\begin{eqnarray*}
\|e_j^T\Delta Q\|_2&\leq & 2\,\sqrt{\ell_j}\>\left(\epsilon^{\perp}_j +
\sqrt{2}\,\mu\,\sr{A}^{1/2}\right)\>\kappa_2(A)+\mathcal{O}(\mu^2), \qquad 1\leq j\leq m.
\end{eqnarray*}
At last, substituting the above into (\ref{e_pertq})
and focussing on the first order terms in $\mu=\epsilon^{\perp}_F$ gives Theorem~\ref{t_6}. 

\subsection{Proof of Theorem~\ref{t_8}}\label{s_t8proof}
Write the perturbations (\ref{e_pertrow}) as $\delA = DA$, where $D$ is a diagonal matrix with
diagonal elements $D_{jj}=\zeta_j\eta_j$, $1\leq j\leq m$. By assumption
$\|\delA\|_2\|A^{\dagger}\|_2\leq \eta\>\kappa_2(A)<1$, so that $\rank(A+\delA)=n$.

As in the proof of Theorem~\ref{t_5}, we start with  (\ref{e_pertq}). To derive bounds
for $\|e_j^T\Delta Q\|_2$ in terms of $\eta_j$ and $\eta$,
represent the perturbed matrix by  
$$A(t)\equiv A+\frac{t}{\eta}\>\delA, \qquad  0\leq t\leq\eta.$$
Let $A(t)=Q(t)R(t)$ be a thin QR decomposition, where 
$Q\equiv Q(0)$, $R\equiv R(0)$, $Q+\Delta Q\equiv Q(\eta)$ and 
$R+\Delta R \equiv R(\eta)$. The derivative of $R$ with respect to $t$ is $\dot{R}$.

Theorem~\ref{t_7} implies 
$$\Delta Q =\delA\,R^{-1} - \epsilon\>Q\,\dot{R}\,R^{-1}+\mathcal{O}(\eta^2)
=DQ-Q\,\dot{R}\,R^{-1}+\mathcal{O}(\eta^2).$$
With $\delA =DA$ this gives
$$e_j^T\,\Delta Q = \eta_j\> e_j^TQ + \eta \>e_j^TQ\,\dot{R}\,R^{-1} + \mathcal{O}(\eta^2),
\qquad 1\leq j \leq m.$$
Taking norms gives 
\begin{eqnarray}\label{e_pertq3}
\|e_j^T\,\Delta Q\|_2\leq \sqrt{\ell_j}\>\left(\eta_j +\eta \>\|\dot{R}\,R^{-1}\|_2\right)+
\mathcal{O}(\eta^2), \qquad 1\leq j\leq m.
\end{eqnarray}
From (\ref{e_rr}) follows
\begin{eqnarray*}
\|\dot{R}\,R^{-1}\|_2&\leq &\norm{\dot{R}R^{-1}}_F 
\leq \frac{\sqrt{2}}{\eta}\>\norm{Q^T\delA  R^{-1}}_F
=\frac{\sqrt{2}}{\eta}\>\norm{Q^TDQ}_F\\
&\leq & \frac{\sqrt{2}}{\eta}\>\norm{Q^T}_F\>\norm{DQ}_2\leq \sqrt{2}n.
\end{eqnarray*}
Combining this with (\ref{e_pertq3}) yields
$$\|e_j^T\,\Delta Q\|_2\leq \sqrt{\ell_j}\>\left(\eta_j +\sqrt{2}\,n\>\eta\right)+
\mathcal{O}(\eta^2), \qquad 1\leq j\leq m.$$
Inserting the above into (\ref{e_pertq}) and focussing on the first order terms in $\eta$
gives Theorem~\ref{t_8}.

%\bibliography{pert}

\end{document}